\documentclass[12pt,reqno]{amsart}
\usepackage[cp1251]{inputenc}
\usepackage[english]{babel}
\usepackage{amssymb}
\usepackage{amsxtra}
\usepackage{mathrsfs}
\usepackage[all]{xy}
\usepackage{cite}
\usepackage{paralist}
\usepackage[hypertex]{hyperref}
\usepackage[active]{srcltx} 
\theoremstyle{plain}
\newtheorem{thm}{Theorem}
\newtheorem{lm}{Lemma}

\newtheorem{pr}{Proposition}

\theoremstyle{definition}

\newcommand*{\wt}{\widetilde}

\newcommand{\ad}{\mathop{\mathrm{ad}}\nolimits}

\newcommand*{\lla}{\longleftarrow}

\newcommand{\CC}{\mathbb{C}}

\newcommand{\R}{\mathbb{R}}
\newcommand{\Z}{\mathbb{Z}}
\newcommand{\N}{\mathbb{N}}

\newcommand*{\cO}{\mathcal{O}}

\newcommand*{\fa}{\mathfrak{a}}

\newcommand*{\ff}{\mathfrak{f}}
\newcommand*{\fg}{\mathfrak{g}}

\newcommand*{\rT}{\mathrm T}

\renewcommand{\le}{\leqslant}
\renewcommand{\ge}{\geqslant}

\let \al         =\alpha
\let \be         =\beta

\let \ep         =\varepsilon      

\let \la         =\lambda
\let \si         =\sigma

\let \Om        =\Omega

\let \phi         =\varphi

\begin{document}
\title{Sheaves of noncommutative smooth and holomorphic functions
associated with the non-Abelian two-dimensional Lie algebra}

\author{O.\,Yu.~Aristov}
\email{aristovoyu@inbox.ru}
\thanks{This work was supported by the Russian Foundation for Basic Research under grant 19-01-00447.}

\begin{abstract}
Dosi and, quite recently, the author showed that, on the character space of a nilpotent Lie
algebra, there exists a sheaf of Fr\'echet--Arens--Michael algebras (of noncommutative holomorphic
functions in the complex case and of noncommutative smooth functions in the real case). We construct similar sheaves (both versions, holomorphic and smooth) on a special space of representations
for the Lie algebra of the group of affine transformations of the real line (which is the simplest
nonnilpotent solvable Lie algebra).

\end{abstract}

\maketitle

\markright{Sheaves of noncommutative functions}

With any triangular Lie algebra  $\fg$ over $\R$  a Fr\'echet--Arens--Michael algebra is associated, whose
elements can be regarded as noncommutative~$C^\infty$  functions (see the author's paper~\cite{ArOld}). In\cite{ArOld}, the author considered a local version of this algebra in the case where~$\fg$  is nilpotent (by analogy with
algebras of noncommutative holomorphic functions defined by Dosi in  \cite{Do10B}  and  \cite{Do10A}).  Namely, for each
open subset $V$ of the character space of~$\fg$, the multiplication in the universal enveloping algebra $U(\fg)$ can be extended to a Fr\'eche--Arens--Michael algebra of polynomial growth (it is denoted by $C^\infty_\fg(V)$), and there arises a sheaf of algebras of this type (which are noncommutative if $\fg$  is non-Abelian) \cite[Theorems~5.1 and~5.5]{ArOld}.
 In a similar way, a sheaf of noncommutative holomorphic functions for a nilpotent
Lie algebra $\fg$ over $\CC$  is constructed (see \cite[Sec.\,5.1, pp.\,119--120]{Do10A},  where certain constraints on $\fg$  are
imposed, and~\cite[Sec.\,6]{ArOld},  where the general case is considered).

However, the definition in the form given in~\cite{ArOld} cannot be applied to nonnilpotent triangular Lie
algebras. It is easy to check that, even for the simplest nonnilpotent triangular Lie algebra  $\fa\ff_1$ (of the
group of affine transformations of the real line), the multiplication in $U(\fa\ff_1)$  does not generally extend
to a continuous operation on the required space  \cite[Sec.\,5.2]{ArOld}.  The same is true of noncommutative
holomorphic functions.

In this paper, we propose a method for overcoming this difficulty in the case of $\fa\ff_1$. We consider a
topological space consisting of special indecomposable representations of~$\fa\ff_1$   (the set of characters is
its proper subset) and construct a Fr\'echet--Arens--Michael algebra $C^\infty_{\fa\ff_1}(V)$  of polynomial growth for
each open set $V$ in this space. The proof that it belongs to the required class uses curious algebras of
triangular matrices; see Sec.\,\ref{s:usege}. As a result, we obtain a sheaf (Theorem~\ref{af1PGSh}),  which can be called a sheaf
of noncommutative smooth functions by analogy with the Abelian case. In a similar way, we construct
a sheaf $\mathfrak{F}_\fg(-)$  of noncommutative holomorphic functions for the complexification of the algebra~$\fa\ff_1$ (Theorem~\ref{Faf1Vgen}). The results obtained for this simplest algebra give hope for the existence of sheaves for
any triangular Lie algebras.

\subsection*{On the used classes of algebras}
Polynomial growth is the most important property of both the globally defined algebra $C^\infty_\fg$, where~$\fg$ is a triangular Lie $\R$-algebra, and the locally defined algebras $C^\infty_\fg(V)$, where~$\fg$ is nilpotent. In particular,
$C^\infty_\fg$ has a universal property in the class of Banach algebras of polynomial growth \cite[Theorem 4.4]{ArOld}. This is also true in the case $\fg=\fa\ff_1$.  Our first objective is to specify locally defined $C^\infty$ algebras for this case.
It is natural to expect that they must be of polynomial growth, and this is indeed so (Theorem~\ref{af1CgU}).

Recall that a locally convex algebra (here and in what follows, all algebras are assumed to be
unital) over $\R$ or $\CC$  is called an  \emph{Arens--Michael algebra}  if its topology is determined by a family of
submultiplicative prenorms (i.e., prenorms for which  $\|ab\|\le \|a\|\,\|b\|$).  In other words, it is a projective
limit of Banach algebras (see \cite{X2}).
 Following \cite[Definition~2.6]{ArOld},  we say that an Arens--Michael $\R$-algebra~$B$ has \emph{polynomial growth} if, for any $b\in B$  and any continuous submultiplicative prenorm
$\|\cdot \|$ on~$B$   there exists a $K>0$ and an $\al\ge0$ such that
$$
\|e^{isb}\|\le K (1+|s|)^{\al} \qquad \text{for all $s\in
\R$.}
$$
(here $\|\cdot \|$ is extended to the complexification). In particular, such are the Banach algebras $C^n[a,b]$ of $n$-times continuously differentiable functions for $n\in\Z_+$,  the Arens--Michael algebras $C^\infty[a,b]$ of infinitely differentiable functions, and their generalizations. The basic properties of algebras of
polynomial growth can be found in~\cite{ArOld} and~\cite{ArNew}.

The algebras of noncommutative $C^\infty$  functions defined in \cite{ArOld} were introduced by analogy with the
algebras of noncommutative holomorphic functions considered previously by Dosi in \cite{Do10B} and \cite{Do10A}  for
nilpotent Lie $\CC$-algebras $\fg$,  namely, with the globally defined algebra $\mathfrak{F}_\fg$ and the locally defined algebras
 $\mathfrak{F}_\fg(V)$. In~\cite{ArOld}, the definition of $\mathfrak{F}_\fg$  was generalized to arbitrary finite-dimensional solvable Lie  $\CC$-algebras
(they are automatically triangular). Note that, in the case $\fg=\fa\ff_1$,  the corresponding algebra $\mathfrak{F}_{\fa\ff_1}$  has
the universal property in the class of all Banach algebras, i.e., is the Arens--Michael envelope; however,
this is a mere coincidence (this does not hold for arbitrary $\fg$).

In \cite[Remark 6.8]{ArOld},  it was mentioned that the algebra $\mathfrak{F}_\fg(V)$  for a nilpotent $\fg$ and the algebra $\mathfrak{F}_\fg$ in the
general case are projective limits of Banach algebras satisfying a polynomial identity. Moreover, we can
restrict the set of polynomial identities, because, as essentially shown in  \cite{ArOld},  these algebras are locally
solvable Arens--Michael algebras. (We say that an associative algebra is  \emph{solvable} if it is solvable as a
Lie algebra, and we say that an Arens--Michael algebra is  \emph{locally solvable} if it is a projective limit of
Banach solvable algebras.) Note also that any Banach $\R$-algebra of polynomial growth is solvable.

As in the case of algebras of noncommutative smooth functions, it is natural to expect that the
properties of algebras $\mathfrak{F}_{\fa\ff_1}(V)$ defined in~\eqref{Fiaf1V}   are similar to those in the nilpotent case. Indeed, the
 $\mathfrak{F}_{\fa\ff_1}(V)$  are locally solvable Arens--Michael algebras (see Theorems~\ref{FgV} and~\ref{Faf1Vgen}).

\section{Algebras of triangular matrices}
\label{s:usege}
In this section, we define certain algebras of triangular matrices whose entries belong to the spaces of
infinitely smooth functions and of holomorphic functions (with domains depending on the position in the
matrix) and show that they are Arens-Michael algebras of polynomial growth in the case of the field~$\R$ and locally solvable in the case of~$\CC$.
These algebras play an auxiliary role and are used in the proofs of
the main theorems of Sec.\,\ref{s:shnc},  but they may also be of independent interest.

\subsection*{The $C^\infty$-version: polynomial growth}

Let $\rT_p$ denote the algebra of upper triangular (including diagonal) real matrices of order $p$, where $p\in\N$. In
\cite[Theorem~2.12]{ArOld}, it was shown that, for any $p,m\in \N$ and any open set~$W$ in~$\R^m$, the
algebra $C^\infty(W,\rT_{p})$ of matrix-valued functions is of polynomial growth. This algebra can be identified
with the algebra  $\rT_{p}(C^\infty(W))$ of triangular matrices with entries in $C^\infty(W)$. We need the following
generalization.

We consider tuples
$$
\underline{K}=\{K_{ij}\!:\, 1\le i\le j \le p\}
$$
of compact subsets of~$\R^m$  such that the interior of each $K_{ij}$  is dense. We denote by $\mathcal{C}^m_p$ the family
of all such sets satisfying the condition  $K_{ik}\subset K_{ij}\cap K_{jk}$ for all admissible indices. Given $n\in\Z_+$, by
$\rT_{p}(C^n(\underline{K}))$   we denote the vector space of upper triangular matrices of the form
$$
\begin{pmatrix}
f_{11}&f_{12} &\cdots&&f_{1p} \\
 &f_{22} && &\\
 & & \ddots&&\vdots\\
  & & &f_{p-1,p-1}&f_{p-1,p}\\
   & & &&f_{pp}
\end{pmatrix},
$$
in which the functions $f_{ij}$ belong to different spaces, namely,  $f_{ij}\in C^n(K_{ij})$. The set $K_{ij}$  is allowed to
be empty; in this case, we set $C^n(K_{ij})=0$.  The space $\rT_{p}(C^n(\underline{K}))$ treated as a Cartesian product of
function spaces is a Banach space.

Let $\underline{K}\in\mathcal{C}^m_p$. To simplify notation, for any $ f\in C^n(K_{ij})$ and $g\in C^n(K_{jk})$, we denote the product
of the restrictions of these functions to   $K_{ik}$ by $fg$.  Obviously,
$fg\in C^n(K_{ik})$. We define multiplication
in  $\rT_{p}(C^n(\underline{K}))$  by the standard formula: the product of matrices $f=(f_{ij})$ and $g=(g_{jk})$  is the matrix
with entries  $h_{ik}=\sum_j f_{ij}g_{jk}$. The associativity of multiplication is obvious.

\begin{pr}\label{Tseqco}
If $n\in\Z_+$, $p,m\in\N$ and $\underline{K}\in\mathcal{C}^m_p$, then $\rT_{p}(C^n(\underline{K}))$ is a Banach $\R$-algebra of
polynomial growth.
\end{pr}
\begin{proof}
 It is easy to see that multiplication in $\rT_{p}(C^n(\underline{K}))$   is continuous; therefore, this is a Banach
algebra. Consider the following extension of Banach algebras:
$$
0\lla C^n(K_{11})\times \cdots \times  C^n(K_{pp}) \lla \rT_{p}(C^n(\underline{K})) \lla I\lla 0,
$$
where $I$ is the ideal of matrices with zero diagonal. Obviously, it is split and nilpotent. Since the
class of Banach $\R$-algebras of polynomial growth is stable with respect to the passage to closed
subalgebras and finite Cartesian products \cite[Proposition~2.11]{ArOld} and all $C^n(K_{jj})$ are of polynomial
growth  \cite[Proposition~2.13]{ArOld}, it follows that $C^n(K_{11})\times \cdots \times  C^n(K_{pp})$ is of polynomial growth as well.
Thus, $\rT_{p}(C(\underline{K})) $ is a split nilpotent extension of a Banach $\R$-algebra of polynomial growth; according
to  \cite[Theorem~2.14]{ArOld}, it also has this property.
\end{proof}

Now let $\underline{W}=\{W_{ij}\!:\, 1\le i\le j \le p\}$  be a tuple of open subsets~$\R^m$ such that $W_{ik}\subset W_{ij}\cap W_{jk}$ for all admissible indices. Consider the vector space $\rT_{p}(C^\infty(\underline{W}))$  of upper triangular matrices $(h_{ij})$, where $h_{ij}\in C^\infty(W_{ij})$ for all $i\le j$; we endow it with the same multiplication as above and the product topology.

\begin{thm}\label{corTseqco}
Let $p,m\in\N$, and let $\underline{W}=\{W_{ij}\!:\, 1\le i\le j \le p\}$  be a tuple of open subsets of~$\R^m$ such that $W_{ik}\subset W_{ij}\cap W_{jk}$  for all admissible indices. Then $\rT_{p}(C^\infty(\underline{W}))$ is the projective limit of
Banach algebras of the form  $\rT_{p}(C^n(\underline{K}))$, where $\underline{K}\in\mathcal{C}^m_p$ and $\underline{K}\subset \underline{W}$ (that is, $K_{ij}\subset W_{ij}$ for all
admissible indices) and, therefore, a Fr\'echet--Arens--Michael  $\R$-algebra of polynomial growth.
\end{thm}
\begin{proof}
By Proposition~\ref{Tseqco},  it suffices to prove the first assertion, i.e., that $\rT_{p}(C^\infty(\underline{W}))$ is the projective
limit of Banach algebras of the form $\rT_{p}(C^n(\underline{K}))$, where $\underline{K}\in\mathcal{C}^m_p$ and $\underline{K}\subset \underline{W}$.

Note that the sets $\underline{K}$ of compact subsets satisfying the conditions $\underline{K}\in\mathcal{C}^m_p$ and $\underline{K}\subset \underline{W}$ form a net
with respect to inclusion. Indeed, the first condition is preserved because
$$K_{ik}\cup K'_{ik} \subset (K_{ij}\cup K'_{ij})\cap (K_{jk}\cup K'_{jk})$$ if $K_{ik}\subset K_{ij}\cap K_{jk}$ and $K'_{ik}\subset K'_{ij}\cap K'_{jk}$, and the preservation of the second is obvious.

If $\underline{K}\subset \underline{K}'$ and $n'\ge n$, we have a continuous homomorphism
$$
\rT_{p}(C^{n'}(\underline{K}'))\to\rT_{p}(C^n(\underline{K})).
$$
There arises the projective system of Banach algebras
$$
(\rT_{p}(C^n(\underline{K}))\!: \underline{K}\in\mathcal{C}^m_p,\,\underline{K}\subset \underline{W},\, n\in\Z_+).
$$

On the other hand,  $\rT_{p}(C^\infty(\underline{W}))$ is the projective limit of spaces of the form
$\rT_{p}(C^n(\underline{K}))$, where $\underline{K}\subset \underline{W}$, in the category of Banach spaces. Thus, to complete the proof, it suffices, given any $\underline{K}\subset \underline{W}$, to find a $\underline{K}'\in\mathcal{C}^m_p$ for which $\underline{K}\subset \underline{K}'\subset\underline{W}$.

First, note that, for fixed  $i'$ and $k'$ satisfying the inequality $i'\le k'$ and any $x\in W_{i'k'}$, there exists
a $\underline{K}\subset \underline{W}$  such that $x\in K_{i'k'}$ and  $\underline{K}\in\mathcal{C}^m_p$.
Indeed, we have $x\in W_{i'j}\cap W_{jk'}$  for all $j$ satisfying the
condition $i'\le j\le k'$. Hence there exists a ball $S$ of nonzero radius centered at $x$ which is contained in
all  $W_{i'j}$ and $W_{jk'}$ for these $j$. We set  $K_{ik}\!:=S$ if  $i'\le i\le k\le k'$ and $K_{ik}\!:=\emptyset$ otherwise.
 It is easy to
check that the required condition $K_{ik}\subset K_{ij}\cap K_{jk}$ holds for all admissible indices, i.e., $\underline{K}\in\mathcal{C}^m_p$.

Consider the topological space $\coprod_{i\le j} W_{ij}$ (disjoint union) and its compact subset $\coprod_{i\le j} K_{ij}$.
According to what was proved above, there exists a family $(\underline{K}_\al)$ such that $\bigcup_\al\underline{K}_\al=\underline{W}$ and $\underline{K}_\al\in\mathcal{C}^m_p$  for
each~$\al$. Since $\coprod_{i\le j} K_{ij}$ is compact and each set in $(\underline{K}_\al)$ is a neighborhood, we can assume that this
family is finite. Since $\mathcal{C}^m_p$ is a net, it follows that, passing to the union, we obtain the required $\underline{K}'$.
\end{proof}

\subsection*{The holomorphic version: locally solvable algebras}
Now we turn to algebras of triangular matrices composed of holomorphic functions.

We consider tuples
$\underline{K}=\{K_{ij}\!:\, 1\le i\le j \le p\}$  of compact subsets of
 $\CC^m$ and tuples $$\underline{W}=\{W_{ij}\!:\, 1\le i\le j \le p\}$$ of open subsets of~$\CC^m$  such that $W_{ik}\subset W_{ij}\cap W_{jk}$ for all admissible indices. (The notation $\mathcal{C}^m_p$ is used
for the family of compact subsets satisfying the same conditions as in the real case except the density of
interior.)

Let $A(K_{ij})$ denote the Banach algebra of functions holomorphic inside $K_{ij}$  and continuous on the
closure of $K_{ij}$ (endowed with the max-norm). Then
$\rT_{p}(A(\underline{K}))$ is a Banach space.

\begin{pr}\label{Tseqhol}
If $p,m\in\N$ and $\underline{K}\in\mathcal{C}^m_p$, then $\rT_{p}(A(\underline{K}))$   is a Banach solvable algebra.
\end{pr}
\begin{proof}
Since multiplication in $\rT_{p}(A(\underline{K}))$  is continuous, it follows that this is a Banach algebra. It is
easy to see that the subspace $[\rT_{p}(A(\underline{K})),\rT_{p}(A(\underline{K}))]$  is nilpotent, and hence this algebra is solvable.
\end{proof}

Consider the vector space $\rT_{p}(\cO(\underline{W}))$  of upper triangular matrices $(f_{ij})$  in which every $f_{ij}$ belongs to
the space $\cO(W_{ij})$ of holomorphic functions; we endow it with the same multiplication as above and with
the product topology.

\begin{thm}\label{corTseqho}
Let  $p,m\in\N$, and let $\underline{W}=\{W_{ij}\!:\, 1\le i\le j \le p\}$  be a tuple of open subsets of~$\CC^m$ such
that $W_{ik}\subset W_{ij}\cap W_{jk}$ for all admissible indices. Then  $\rT_{p}(\cO(\underline{W}))$ is the projective limit of Banach
algebras of the form $\rT_{p}(A(\underline{K}))$, where $\underline{K}\in\mathcal{C}^m_p$ and $\underline{K}\subset \underline{W}$; therefore, this is a locally solvable
Fr\'echet--Arens--Michael algebra.
\end{thm}
The proof of this theorem is similar to that of Theorem~\ref{corTseqco}  but uses Proposition~\ref{Tseqhol} instead of
Proposition~\ref{Tseqco}.

\section{Locally defined algebras of noncommutative functions and sheaves}
\label{s:shnc}

Let $\fa\ff_1$ denote the two-dimensional Lie algebra (real or complex, depending on the context) with
basis $e_1,e_2$ and multiplication defined by $[e_1,e_2]=e_2$.
 In this section, we define a topological space
consisting of special indecomposable representations of $\fa\ff_1$  and show that, for each open set $V$, we can
define a Fr\'echet–Arens--Michael algebra $C^\infty_{\fa\ff_1}(V)$ of polynomial growth in the real case and a locally
solvable Fr\'echet–Arens--Michael algebra  $\mathfrak{F}_{\fa\ff_1}(V)$  in the complex case. We also prove the main results
of this paper, namely, that the corresponding functors determine sheaves (Theorems~\ref{af1PGSh} and~\ref{Faf1Vgen}).

\subsection*{The sheaf of noncommutative $C^\infty$ functions}
In the case where $\fg$ is a nilpotent Lie $\CC$-algebra and $V$ is an open subset of its character space,
multiplication in the universal enveloping algebra $U(\fg)$ continuously extends to
$C^\infty_{\fg}(V)\!:=C^\infty(V)^{\Z_+}$ (see \cite[Theorem 5.1]{ArOld}).  For the algebra $\fa\ff_1$, which is not nilpotent, the definition must be modified. The
idea is to consider the product of spaces of the form $C^\infty(V_q)$, where the open sets $V_q$ are subject to certain
relations, instead of the power of $C^\infty(V)$.

Consider the following family of indecomposable representations of the algebra~$\fa\ff_1$:
$$
\si_{r,q}(e_1)\!:=X_q+r\quad\text{and}\quad \si_{r,q}(e_2)\!:=Y_q\qquad(r\in\R,\, q\in\Z_+),
$$
where \begin{equation}\label{XpYp}
X_q\!:=
\begin{pmatrix}
q& &&& \\
 &q-1 && &\\
 & & \ddots&&\\
  & & &1&\\
   & & &&0
\end{pmatrix}\quad\text{and}\quad
Y_q\!:=
\begin{pmatrix}
0& 1&&& \\
&0& 1&& \\
 && \ddots& \ddots&\\
 & &&0 &1\\
  & &&&0
\end{pmatrix}.
\end{equation}
(For brevity, we identify finite-dimensional operators with their matrices.)

Let us denote the set $\{\si_{r,q}\!:\,r\in\R, \,q\in\Z_+\}$ by $\Om_\R$ and endow it with a topology as follows. Consider
the shift operation on the family of all subsets of $\R$:
$$
X+\mu\!:=\{x+\mu\!:\,x\in X\}\qquad(X\subset \R,\,\mu\in\R).
$$
We declare a subset of $\Om_\R$  to be open if it is of the form
\begin{equation}\label{openinOm}
\bigcup_{q=0}^\infty\{\si_{r,q}\!:\,r\in V_q\},
\end{equation}
where $(V_q)_{q\in\Z_+}$ is a sequence of open subsets of $\R$ satisfying the condition
\begin{equation}\label{Ppp1Up}
 V_q\cap (V_q+1)\supset V_{q+1}\qquad(q\in\Z_+);
\end{equation}
in particular, the sequence $(V_q)$ is nonincreasing. (The shift by $1$ appears because this number is an
eigenvalue of the operator $\ad e_1$  in the adjoint representation.) What we obtain is indeed a topology,
because the shift commutes with the intersection and union operations.  (Note that this topology is not
Hausdorff, because any open set containing  $\si_{r,q+1}$ contains also  $\si_{r,q}$ and $\si_{r-1,q}$.)

For an open subset $V$ of $\Om_\R$ represented in the form \eqref{openinOm}, we set
\begin{equation}\label{Ciaf1V}
C^\infty_{\fa\ff_1}(V)\!:=\prod_{q=0}^\infty C^\infty(V_q).
\end{equation}
(Note that if $V_q$  is empty for some $q$, then $C^\infty(V_q)$ is the zero algebra. Obviously, then the algebras for
larger $q$ are zero as well; in this case, the product can be regarded to be finite.)

In the case where $V=\Om_\R$, we have $V_q=\R$  for all $q$. The space $C^\infty_{\fa\ff_1}(\Om_\R)$ coincides with the space $C^\infty_{\fg}$ defined in \cite{ArOld} (for $\fg=\fa\ff_1$).

Writing the elements of the universal enveloping algebra $U(\fa\ff_1)$  in the form $a=\sum_j f_j(e_1)e_2^j$, where $f_j$ belongs to the polynomial algebra $\R[\la]$, we obtain the linear map
\begin{equation}\label{UtoCi}
U(\fa\ff_1)\to C^\infty_{\fa\ff_1}(V)\!:a\mapsto (f_q).
\end{equation}

\begin{thm}\label{af1CgU}
Let $V$ be an open subset of $\Om_\R$.  Then multiplication in $U(\fa\ff_1)$  extends to a continuous
multiplication in $C^\infty_{\fa\ff_1}(V)$. Moreover,  $C^\infty_{\fa\ff_1}(V)$  with this multiplication is a projective limit of
Banach $\R$-algebras of polynomial growth and, therefore, a Fr\'echet--Arens--Michael algebra of
polynomial growth.
\end{thm}
\begin{proof}
We argue as in \cite[Example~4.7]{ArOld} with the only difference that, instead of $\rT_{q+1}(C^\infty(\R))$  we use
an algebra of the form $\rT_{q+1}(C^\infty(\underline{W}))$ ) with a special $\underline{W}$.

Let $\R[\la;\rT_{q+1}]$  denote the algebra of matrix-valued polynomials. For each $q\in\Z_+$, consider the
homomorphism
\begin{equation*}
\wt\pi_q\!:U(\fa\ff_1)\to \R[\la;\rT_{q+1}],\quad \wt\pi_q(a)\!:=(\la\mapsto\si_{\la,q}(a)).
\end{equation*}
According to the Poincar\'e-Birkhoff-Witt theorem, any element of $U(\fa\ff_1)$ can be represented in the form $a=\sum_j f_j(e_1)e_2^j$.
 It is easy to check that
\begin{equation}\label{tildfull2d}
[\wt\pi_q(a)](\la)=\begin{pmatrix}
f_0(\la+q)&f_1(\la+q) &\cdots&&f_q(\la+q) \\
 &f_0(\la+q-1) && &\\
 & & \ddots&&\vdots\\
  & & &f_0(\la+1)&f_1(\la+1)\\
   & & &&f_0(\la)
\end{pmatrix}.
\end{equation}

We set
\begin{equation}\label{Wq+1}
W^{q+1}_{ij}\!:=V_{j-i}-q-1+i\qquad(1\le i\le j\le q+1),
\end{equation}
and
$$
\underline{W}_{(q+1)}\!:=\{W^{q+1}_{ij}\!:\, 1\le i\le j \le q+1\}.
$$
Obviously $\underline{W}_{(q+1)}$  is a tuple of open subsets of~$\R$.

Let us check that $W^{q+1}_{ik}\subset W^{q+1}_{ij}\cap W^{q+1}_{jk}$ for $1\le i\le j \le k\le q+1$.
It follows from \eqref{Ppp1Up} that $V_{k-i}\subset V_{j-i}$ and  $V_{k-i} \subset V_{k-j}+j-i$ (because $k-j=(k-i)-(j-i)$). Therefore,
$$
W^{q+1}_{ik}=V_{k-i}-q-1+i\subset (V_{j-i}-q-1+i)\cap ( V_{k-j}-q-1+j) = W^{q+1}_{ij}\cap W^{q+1}_{jk}.
$$

Thus, $\rT_{q+1}(C^\infty(\underline{W}_{(q+1)}))$ satisfies the assumptions of Theorem~\ref{corTseqco} and, therefore, is a projective limit
of Banach $\R$-algebras of polynomial growth and a Fr\'echet algebra.

Increasing the range of each $\wt\pi_q$ to $\rT_{q+1}(C^\infty(\underline{W}_{(q+1)}))$, we can consider the homomorphism
\begin{equation*}
\rho_V\!:U(\fa\ff_1)\to \prod_{q=0}^\infty \rT_{q+1}(C^\infty(\underline{W}_{(q+1)})) \!: a\mapsto (\wt\pi_q(a)).
\end{equation*}
Note that $U(\fa\ff_1)$  is a dense subspace o $C^\infty_{\fa\ff_1}(V)$ (because, by Weierstrass' theorem on approximation
of continuously differentiable functions, the polynomial algebra is dense in each $C^\infty(V_q)$).
Since the
class of Fr\'echet algebras is stable with respect to the passage to countable Cartesian products and
closed subalgebras and the same is true of the class of projective limits of Banach algebras of polynomial
growth \cite[Proposition~2.11]{ArOld},  it follows that, to complete the proof, it suffices to show that $\rho_V$ is continuous and topologically injective (with respect to the restriction of the topology of  $C^\infty_{\fa\ff_1}(V)$ to $U(\fa\ff_1)$).

First, we prove the continuity. To this end, we need to show that the homomorphism
 $U(\fa\ff_1)\to  \rT_{q+1}(C^\infty(\underline{W}_{(q+1)}))$ is continuous for each~$q$.
Recall that the topology of $C^\infty(U)$, where $U$ is an open subset of $\R^m$, is
determined by the prenorms
\begin{equation*}
|f|_{K,n}\!:=\sum_{\be\in \Z_+^m,\,|\be|=n}|f^{(\be)}|_{K,0}\,,\qquad
\text{where}\quad |f|_{K,0}\!:=\max_{x\in K}|f(x)|,
\end{equation*}
$n\in\Z_+$ and $K$ is any compact subset of~$U$. (Here $|\be|\!:=\be_1+\cdots \be_m$ for $\be=(\be_1,\ldots, \be_m)\in\Z_+^m$ and~$f^{(\be)}$ denotes the partial derivative of the function~$f$.)

On the one hand, the topology of $\rT_{q+1}(C^\infty(\underline{W}_{(q+1)}))$  is determined by the prenorms
\begin{equation}\label{Kijnfaf1}
h\mapsto |h_{ij}|_{K_{ij},n},\qquad h=(h_{ij})\in \rT_{q+1}(C^\infty(\underline{W}_{(q+1)})),
\end{equation}
where $1\le i\le j\le q+1$, $K_{ij}$ is a compact subset of $W_{ij}$, and $n\in\Z_+$. On the other hand, the topology
of $U(\fa\ff_1)$, which is inherited from $C^\infty_{\fa\ff_1}(V)$, is determined by the prenorms
\begin{equation}\label{qMlfaf1}
a\mapsto |f_p|_{M,l}\qquad\text{($p,l \in \Z_+$,   $M\subset V_p$  is compact)}.
\end{equation}

Estimating the matrix elements in~\eqref{tildfull2d}, we obtain
$$
|\wt\pi_q(a)_{ij}|_{K_{ij},n}\le |f_{j-i}|_{K_{ij}+q+1-i,\,n}\qquad(i\le j)
$$
for each $a\in U(\fa\ff_1)$. Therefore, to prove the continuity of  $\rho_V$, it suffices, given any set $\underline{K}$ of
compact subsets such that  $\underline{K}\subset\underline{W}_{(q+1)}$, to find compact sets  $M_{k}\subset V_{k}$, $k=0,\ldots,q$, such that $K_{ij}+q+1-i\subset M_{j-i}$ for $i\le j$.
Since
$$
K_{ij}+q+1-i\subset W^{q+1}_{ik}+q+1-i=V_{j-i},
$$
it follows that the $M_k\!:=\bigcup_{j-i=k}K_{ij}$  are as required. This proves the continuity.

Now we show that $\rho_V$ is topologically injective. To this end, we must estimate each prenorm in~\eqref{qMlfaf1} (for given $p,l\in \Z_+$ and a compact subset $M$ of~$V_p$) in terms of prenorms of the form~\eqref{Kijnfaf1}.  According
to \eqref{Wq+1},  we have $W^{p+1}_{1,p+1}\!:=V_p-p$. Therefore, $M-p \subset W^{p+1}_{1,p+1}$. Now consider the homomorphism $\wt\pi_p$.
Since  $f_p$ corresponds to the upper right corner entry in~\eqref{tildfull2d} (with $q=p$),  it follows that
\begin{equation*}
|f_p|_{M,l}=|(\wt\pi_p(a))_{1,p+1}|_{M-p,l}.
\end{equation*}
Hence $\rho$ is topologically injective, which completes the proof of the theorem.
\end{proof}

Recall that a \emph{sheaf of Fr\'echet spaces (Fr\'echet algebras)} is defined as the presheaf of  Fr\'echet spaces (Fr\'echet algebras)  which turns into a sheaf of sets after the application of the forgetting functor (see
the discussion in \cite[Remark~5.4]{ArOld}).  In what follows, sheaves in subcategories of Fr\'echet algebras are
understood in the same sense.

Let $V$ and $W$ be open subsets of $\Om_\R$, and let $W\subset V$. Then the topology on $U(\fa\ff_1)$, which
is determined by the embedding $U(\fa\ff_1)\to C^\infty_{\fa\ff_1}(V)$, is finer than that determined by the embedding
 $U(\fa\ff_1)\to C^\infty_{\fa\ff_1}(W)$.  Since both algebras are completions of $U(\fa\ff_1)$, we have a continuous homomorphism $\tau_{VW}\!:C^\infty_{\fa\ff_1}(V)\to C^\infty_{\fa\ff_1}(W)$.

\begin{thm}\label{af1PGSh} \emph{(cf. \cite[Theorem 5.5]{ArOld})}
The correspondences
$$
V\mapsto C^\infty_{\fa\ff_1}(V)\quad\text{and}\quad (W\subset V)\mapsto \tau_{VW}
$$
determine a sheaf of Fr\'echet--Arens--Michael $\R$-algebras of polynomial growth on $\Om_\R$.
\end{thm}
\begin{proof}
The correspondence under consideration is a contravariant functor from the category of open
subsets of~$\Om_\R$  to the category of  Fr\'echet--Arens--Michael $\R$-algebras of polynomial growth,  i.e., we
have a presheaf. By definition (see \eqref{Ciaf1V})  for each open set $V\subset \Om_\R$,  the Fr\'echet space $C^\infty_{\fa\ff_1}(V)$  is the
countable product of the spaces $C^\infty(V_p)$.  The gluing axiom for $C^\infty$ functions on~$\R$ implies the gluing
axiom for the presheaf of sets on~$\Om_\R$.
\end{proof}

\subsection*{The sheaf of noncommutative holomorphic functions}

In this section, we consider the complexification of the real Lie algebra discussed above; to simplify
notation, we denote it by the same symbol $\fa\ff_1$. The basic constructions for it are the same as in the
real case. In particular, we need the family $\{\si_{\la,q}\!:\,\la\in\CC, \,q\in\Z_+\}$ of representations of the form~\eqref{XpYp}; we
denote it by $\Om_\CC$. We declare a subset of $\Om_\CC$ to be open if it has the form
\eqref{openinOm},
where $(V_q)_{q\in\Z_+}$ is a sequence
of open subsets of $\CC$ satisfying condition~\eqref{Ppp1Up}.

An important difference from the case of $C^\infty$  functions is that the polynomial algebra is not
necessarily dense in the algebra of holomorphic functions on an open set. However, this constraint
is of technical character and can easily be overcome.

For an open set $V$ in $\Om_\CC$ of the form~\eqref{openinOm}, we put
\begin{equation}\label{Fiaf1V}
\mathfrak{F}_{\fa\ff_1}(V)\!:=\prod_{q=0}^\infty \cO(V_q),
\end{equation}
where $\cO(V_q)$ denotes the algebra of holomorphic functions. In particular, the Fr\'echet space $\mathfrak{F}_{\fa\ff_1}(\Om_\CC)$ coincides with the space $\mathfrak{F}_{\fa\ff_1}$  defined in \cite[Sec.\,6]{ArOld}.

Let $\mathcal{B}$ denote the family of all open subsets $V=(V_q)$ of~$\Om_\CC$ satisfying the following conditions:

(1)~there is a $p\in\Z_+$  such that the sets $V_q$  are empty if $q>p$;

(2)~for each $q=0,\ldots, p$, there exists an $\ep\in(0,1)$ such that $V_q$  is the union of open disks of radius $\ep$ centered at $\si_{\la,q},\ldots,\si_{\la-p+q,q}$   (in particular, $V_p$ is a disk centered a $\si_{\la,p}$).

It is easy to see that
$\mathcal{B}$  is a base for the topology of $\Om_\CC$.

\begin{lm}
If $V$ belongs to $\mathcal{B}$,  then the image of the linear map $U(\fa\ff_1)\to \mathfrak{F}_{\fa\ff_1}(V)$ defined by~\eqref{UtoCi}  is
dense.
\end{lm}
\begin{proof}
 It suffices to show that the image of the map $\CC[\la]\to \cO(V_q)$ is dense for each $q$. Since $V$ belongs
to $\mathcal{B}$, it follows that $V_q$ is either the empty set or a union of disjoint open disks in~$\CC$. Thus, the required
assertion readily follows from Mergelyan's theorem (see, e.g., \cite[p.\,386, Theorem 20.5]{Ru87}).
\end{proof}

The three theorems concluding the paper are similar to Theorems~6.2, 6.5, and~6.6 in~\cite{ArOld}.

\begin{thm}\label{FgV}
Let $V=(V_q)$ be an open subset of $\Om_\CC$ such that the polynomial algebra is dense
in  $\cO(V_q)$  for each~$q$ (e.g., $V\in\mathcal{B}$).
Then the multiplication in $U(\fa\ff_1)$ can be extended to a
continuous multiplication in $\mathfrak{F}_{\fa\ff_1}(V)$. With respect to this multiplication,  $\mathfrak{F}_{\fa\ff_1}(V)$is a locally
solvable Fr\'echet--Arens--Michael $\CC$-algebra.
\end{thm}
We need the following proposition, which is verified directly.
\begin{pr}\label{stlosol}
The class of locally solvable Arens--Michael algebras is stable with respect to the
passage to Cartesian products and closed subalgebras.
\end{pr}

\begin{proof}[Proof of Theorem~\ref{FgV}]
This theorem is proved by the same argument as Theorem~\ref{af1CgU} with minor modifications. We consider the homomorphism
\begin{equation*}
U(\fa\ff_1)\to \prod_{q=0}^\infty \rT_{q+1}(\cO(\underline{W}_{(q+1)})) \!: a\mapsto (\wt\pi_q(a)),
\end{equation*}
where $\underline{W}_{(q+1)}$  is defined in the same way as previously. To check that the algebra $\rT_{q+1}(\cO(\underline{W}_{(q+1)}))$ is a locally solvable Fr\'echet--Arens--Michael algebra, we use Theorem~\ref{corTseqho}  instead of Theorem~\ref{corTseqco}. By
assumption, the homomorphism has dense image, and according to Proposition~\ref{stlosol}  it suffices to check
the continuity and topological injectivity. These properties can be proved in the same way as in Theorem~\ref{af1CgU}, because all estimates used in the proof involve only uniform norms (without derivatives).
\end{proof}

If $V$ and $W$ belong to $\mathcal{B}$ and $W\subset V$, then we have the continuous homomorphism
$\tau_{VW}\!:\mathfrak{F}_{\fa\ff_1}(V)\to \mathfrak{F}_{\fa\ff_1}(W)$.

\begin{thm}\label{nilShD}
On $\Om_\CC$,  the correspondences
$$
V\mapsto \mathfrak{F}_{\fa\ff_1}(V)\quad\text{and}\quad (W\subset V)\mapsto \tau_{VW}
$$
determine a sheaf of locally solvable Fr\'echet--Arens--Michael $\CC$-algebras on the topology
base $\mathcal{B}$.
\end{thm}
The proof is similar to that of Theorem~\ref{af1PGSh}  with the obvious replacement of the sheaf of smooth
functions by the sheaf of holomorphic functions and of the family of all open sets by the topology base.

\begin{thm}\label{Faf1Vgen}
For each open subset
$V$ of $\Om_\CC$, there exists a multiplication with respect to which
 $\mathfrak{F}_{\fa\ff_1}(V)$  is a locally solvable Fr\'echet--Arens--Michael algebra. In the case where the polynomial
algebra is dense in $\cO(V)$,  this multiplication is a continuous extension of that in $U(\fa\ff_1)$.
Moreover, the correspondences in Theorem~\ref{nilShD}
determine the sheaf $\mathfrak{F}_{\fa\ff_1}(-)$  of locally solvable
Fr\'echet--Arens--Michael $\CC$-algebras on~$\Om_\CC$.
\end{thm}
\begin{proof}
Arguing as in the proof of Theorem~6.6 in~\cite{ArOld}, we show that $\mathfrak{F}_{\fa\ff_1}(-)$ is an extension of the sheaf on
the base $\mathcal{B}$  constructed in Theorem~\ref{nilShD}. In particular, it turns out to be a sheaf of Fr\'echet--Arens--Michael
algebras. The local solvability follows from Proposition~\ref{stlosol}.
\end{proof}

\end{document}